\documentclass[12pt,a4paper]{article}
\usepackage[utf8x]{inputenc}
\usepackage{ucs}
\usepackage{amsmath}
\usepackage{amsfonts}
\usepackage{amssymb}
\usepackage{amsthm}
\usepackage{mathrsfs}
\usepackage{xcolor}
\usepackage[top=2cm,bottom=2cm,left=3cm,right=2cm]{geometry}
\usepackage[dvipdfm,pdfstartview=FitH,colorlinks=true,linkcolor=blue,urlcolor=orange,%
filecolor=green,citecolor=red]{hyperref}

\title{A curious formula related to the Euler Gamma function}
\author{Bakir FARHI \\
Department of Mathematics \\
University of B\'ejaia \\
Algeria \\[1mm]
\href{mailto:bakir.farhi@gmail.com}{bakir.farhi@gmail.com}}
\date{}

\newtheorem{thm}{Theorem}

\newtheorem{coll}[thm]{Corollary}
\newtheorem{lemma}[thm]{Lemma}

\def\N{\mathbb{N}}
\def\Z{\mathbb{Z}}

\def\R{\mathbb{R}}
\def\C{\mathbb{C}}

\def\modd#1 #2{#1\ ({\rm mod}\ #2)}
\def\vabs#1{\left|#1\right|}

\begin{document}
\maketitle

\vspace*{-8cm}
\noindent{\large\bf To appear}
\vspace*{8cm}

\begin{abstract}
\noindent In this note, we prove that for all $x \in (0 , 1)$, we have:
$$
\log\Gamma(x) ~=~ \frac{1}{2} \log\pi + \pi \boldsymbol{\eta} \left(\frac{1}{2} - x\right) - \frac{1}{2} \log\sin(\pi x) + \frac{1}{\pi} \sum_{n = 1}^{\infty} \frac{\log n}{n} \sin(2 \pi n x) ,
$$
where $\Gamma$ denotes the Euler Gamma function and 
$$
\boldsymbol{\eta} ~:=~ 2 \int_{0}^{1} \log\Gamma(x) \cdot \sin(2 \pi x) \, d x ~=~ 0.7687478924\dots
$$
\end{abstract}
\textit{MSC 2010}: Primary 33B15, 42A16.  \\
\textit{Keywords}: The Euler Gamma function, Special functions, Fourier expansions.

\section{Introduction}
Throughout this note, we let $\boldsymbol{\eta} := 2 \int_{0}^{1} \log\Gamma(x) \cdot \sin(2 \pi x) \, d x$ and we let $\langle x\rangle$ denote the fractional part of a given real number $x$. We let also $B_k$ ($k \in \N$) denote the Bernouilli polynomials which are usually defined by means of the generating function:
$$
\frac{t e^{x t}}{e^t - 1} ~=~ \sum_{k = 0}^{\infty} B_k(x) \frac{t^k}{k!} .
$$
We just note that each polynomial $B_k$ is monic and has degree $k$ and that the two first Bernouilli polynomials are $B_0(x) = 1$ and $B_1(x) = x - \frac{1}{2}$.

The Euler Gamma function is one of the most important special functions in complex analysis. For complex numbers $z$ with positive real part, it can be defined by the convergent improper integral:
$$
\Gamma(z) := \int_{0}^{+ \infty} t^{z - 1} e^{-t} \, d t .
$$
From this, we easily deduce the important functional equation $\Gamma(z + 1) = z \Gamma(z)$ which we use to extend $\Gamma$ by analytic continuation to all complex numbers except the non-positive integers. The obtained function is meromorphic with simple poles at the non-positive integers; that is what we call ``the Gamma function'' in its generality. We also precise that the Gamma function is an interpolation of the sequence ${((n - 1)!)}_{n \geq 1}$; precisely, we have $\Gamma(n) = (n - 1)!$ for all positive integer $n$. Actually, the interpolation of the factorial sequence is the motivation of Euler that lead him to discover the Gamma function. More interestingly, the Bohr-Mollerup theorem states that the Gamma function is the unique function $f$ satisfying $f(1) = 1$ and $f(x + 1) = x f(x)$, which is log-convex on the positive real axis.

Several important formulas and properties are known for the Gamma function. We recommend the reader to consult the books \cite{and,art,god}. In this note, we just cite the formulas we use for proving our main result. The first one is the so-called ``Euler's reflection formula'', which is given (for all $z \in \C \setminus \Z$) by:
\begin{equation}\label{eq1}
\Gamma(z) \Gamma(1 - z) = \frac{\pi}{\sin(\pi z)}
\end{equation}
and the second one is ``the Gauss multiplication formula'' which is given by:
\begin{equation}\label{eq2}
\Gamma(z) \Gamma\left(z + \frac{1}{k}\right) \cdots \Gamma\left(z + \frac{k - 1}{k}\right) = (2 \pi)^{\frac{k - 1}{2}} k^{- k z + \frac{1}{2}} \Gamma(k z)
\end{equation}
and holds for any positive integer $k$ and any complex number $z$ for which the both sides of \eqref{eq2} are well-defined. The particular case $k = 2$ of \eqref{eq2} is the so-called ``Legendre's formula of duplication''.

The aim of this note is to obtain the Fourier expansion of the $1$-periodic function which coincides with the function $\log\Gamma$ on the interval $(0 , 1)$.
\subsection*{Some useful Fourier expansions}
The proof of our main result needs the Fourier expansion of the two $1$-periodic functions $x \mapsto \log\vabs{\sin(\pi x)}$ and $x \mapsto B_1(\langle x\rangle)$. These Fourier expansions are both known. \\
The Fourier expansions of the functions $x \mapsto B_k(\langle x\rangle)$ ($k \geq 1$) was discovered by Hurwitz in 1890 and are given by the identity:  
$$
B_k(\langle x\rangle) ~=~ - \frac{2 \cdot k!}{(2 \pi)^k} \sum_{n = 1}^{\infty} \frac{\cos\left(2 \pi n x - \frac{k \pi}{2}\right)}{n^k} ,
$$
which holds for all $x \in \R \setminus \Z$ if $k = 1$ and for all $x \in \R$ if $k \geq 2$. \\
Taking $k = 1$ in the last identity, we obtain in particular for all $x \in \R \setminus \Z$:
\begin{equation}\label{fexpb1}
B_1(\langle x\rangle) ~=~ \langle x\rangle - \frac{1}{2} ~=~ - \frac{1}{\pi} \sum_{n = 1}^{\infty} \frac{\sin(2 \pi n x)}{n}  
\end{equation}
On the other hand, the Fourier expansion of the $1$-periodic function $x \mapsto \log\vabs{\sin(\pi x)}$ is well-known and given (for all $x \in \R \setminus \Z$) by:
\begin{equation}\label{fexplogsin}
\log\vabs{\sin(\pi x)} ~=~ - \log 2 - \sum_{n = 1}^{\infty} \frac{\cos(2 \pi n x)}{n}
\end{equation}
For all these Fourier expansions and tigonometric series in general, we recommend the reader to consult the book of A. Zygmund \cite{zyg}. 
\section{The results}
Our main result is the following:
\begin{thm}\label{t1}
For all $x \in (0 , 1)$, we have:
$$
\log\Gamma(x) ~=~ \frac{1}{2} \log\pi + \pi \boldsymbol{\eta} \left(\frac{1}{2} - x\right) - \frac{1}{2} \log\sin(\pi x) + \frac{1}{\pi} \sum_{n = 1}^{\infty} \frac{\log n}{n} \sin(2 \pi n x) ,
$$
where $\boldsymbol{\eta} := 2 \int_{0}^{1} \log\Gamma(x) \cdot \sin(2 \pi x) \, d x = 0.7687478924\dots$.
\end{thm}
The proof of Theorem \ref{t1} requires the following lemma:
\begin{lemma}\label{l1}
For all positive integers $n$ and $k$ and all real number $x$, we have:
$$
\sum_{\ell = 0}^{k - 1} \cos\left\{2 \pi n \left(x + \frac{\ell}{k}\right)\right\} ~=~ \begin{cases}
0 & \text{if}~~ n \not\equiv \modd{0} {k} \\
k \cos(2 \pi n x) & \text{if}~~ n \equiv \modd{0} {k}
\end{cases}
$$
and 
$$
\sum_{\ell = 0}^{k - 1} \sin\left\{2 \pi n \left(x + \frac{\ell}{k}\right)\right\} ~=~ \begin{cases}
0 & \text{if}~~ n \not\equiv \modd{0} {k} \\
k \sin(2 \pi n x) & \text{if}~~ n \equiv \modd{0} {k}
\end{cases} .
$$
\end{lemma}
\begin{proof}
Let $n$, $k$ be positive integers and $x$ be a real number. Define
$$
A ~:=~ \sum_{\ell = 0}^{k - 1} \cos\left\{2 \pi n \left(x + \frac{\ell}{k}\right)\right\} ~~\text{and}~~ B ~:=~ \sum_{\ell = 0}^{k - 1} \sin\left\{2 \pi n \left(x + \frac{\ell}{k}\right)\right\} . 
$$
Then, we have:
\begin{eqnarray*}
A + i B ~=~ \sum_{\ell = 0}^{k - 1} e^{2 \pi n \left(x + \frac{\ell}{k}\right) i} & = & e^{2 \pi n i x} \sum_{\ell = 0}^{k - 1} \left(e^{\frac{2 \pi n}{k} i}\right)^{\ell} \\
& = & e^{2 \pi n i x} \times \begin{cases}
\frac{1 - \left(e^{\frac{2 \pi n}{k} i}\right)^k}{1 - e^{\frac{2 \pi n}{k} i}} & \text{if}~~ n \not\equiv \modd{0} {k} \\
k & \text{if}~~ n \equiv \modd{0} {k}
\end{cases} .
\end{eqnarray*}
Thus:
$$
A + i B ~=~ \begin{cases}
0 & \text{if}~~ n \not\equiv \modd{0} {k} \\
k e^{2 \pi n i x} & \text{if}~~ n \equiv \modd{0} {k} 
\end{cases} .
$$
The identities of the lemma follow by identifying the real and the imaginary parts of the two sides of the last identity. The lemma is proved.
\end{proof}
\subsection*{Proof of Theorem \ref{t1}:}
For the following, we let $\varphi$ denote the $1$-periodic function which coincides with the function $\log\Gamma$ on the interval $(0,1)$. We let also
$$
a_0 + \sum_{n = 1}^{\infty} a_n \cos(2 \pi n x) + \sum_{n = 1}^{\infty} b_n \sin(2 \pi n x) 
$$
denote the Fourier series associated to $\varphi$ (where $a_n , b_n \in \R$). So, we have\\ $a_0 = \int_{0}^{1} \log\Gamma(x) \, d x$ and for all positive integer $n$:
$$
a_n ~=~ 2 \int_{0}^{1} \log\Gamma(x) \cdot \cos(2 \pi n x) \, d x  ~~~~\text{and}~~~~
b_n ~=~ 2 \int_{0}^{1} \log\Gamma(x) \cdot \sin(2 \pi n x) \, d x .
$$
Remark that all these improper integrals converges because the function $\log\Gamma$ is continuous on $(0,1]$ and, at the neighborhood of $0$, we have that $\log\Gamma(x) = \log\Gamma(x + 1) - \log x \sim_{0} - \log x$ and that $\int_{0}^{1} \log x \, d x$ converges. In addition, because $\varphi$ is of class $\mathscr{C}^{1}$ on all the intervals constituting $\R \setminus \Z$, we have according to the classical Dirichlet theorem:
\begin{equation}\label{eq3}
\varphi(x) ~=~ a_0 + \sum_{n = 1}^{\infty} a_n \cos(2 \pi n x) + \sum_{n = 1}^{\infty} b_n \sin(2 \pi n x) 
\end{equation}
(for all $x \in \R \setminus \Z$). Now, we are going to calculate the $a_n$'s and the $b_n$'s. \\[1mm]
\textbf{Calculation of the $a_n$'s:}\\[1mm]
To do this, we lean on the Euler formula \eqref{eq1} and on the Fourier expansion of the function $x \mapsto \log|\sin(\pi x)|$ given by \eqref{fexplogsin}. Using \eqref{eq3}, we have for all $x \in \R \setminus \Z$:
\begin{equation}\label{eq4}
\varphi(x) + \varphi(1 - x) ~=~ 2 a_0 + \sum_{n = 1}^{\infty} 2 a_n \cos(2 \pi n x)
\end{equation}
But, on the other hand, according to \eqref{eq1} and \eqref{fexplogsin}, we have for all $x \in (0 , 1)$:
$$
\varphi(x) + \varphi(1 - x) ~=~ \log\Gamma(x) + \log\Gamma(1 - x) ~=~ \log\pi - \log\sin(\pi x) ~=~ \log\pi - \log\left|\sin(\pi x)\right| .
$$
But since the two functions $x \mapsto \varphi(x) + \varphi(1 - x)$ and $x \mapsto \log\pi - \log\left|\sin(\pi x)\right|$, which are defined on $\R \setminus \Z$, are both $1$-periodic, we generally have for all $x \in \R \setminus \Z$:
$$
\varphi(x) + \varphi(1 - x) ~=~ \log\pi - \log\left|\sin(\pi x)\right| .
$$
It follows, according to \eqref{fexplogsin}, that for all $x \in \R \setminus \Z$, we have:
\begin{equation}\label{eq5}
\varphi(x) + \varphi(1 - x) ~=~ \log(2 \pi) + \sum_{n = 1}^{\infty}\frac{\cos(2 \pi n x)}{n}
\end{equation}
The comparison between \eqref{eq4} and \eqref{eq5} gives (according to the uniqueness of the Fourier expansion):
\begin{equation}\label{coeff-a_n}
a_0 ~=~ \frac{1}{2} \log(2 \pi) ~~~~\text{and}~~~~ a_n ~=~ \frac{1}{2 n} ~~~~ (\forall n \geq 1)
\end{equation} 
\textbf{Calculation of the $b_n$'s:}\\[1mm]
To do this, we lean on the one hand on Lemma \ref{l1} and on the other hand on the Gauss formula \eqref{eq2} and on the Fourier expansion of the function $x \mapsto B_1(\langle x\rangle)$. \\
Let $k$ be a positive integer and let $\frac{1}{k} \Z$ denote the set $\{\frac{n}{k} : n \in \Z\}$. From \eqref{eq3}, we have for all $x \in \R \setminus \frac{1}{k} \Z$:
\begin{multline*}
\varphi(x) + \varphi\left(x + \frac{1}{k}\right) + \dots + \varphi\left(x + \frac{k - 1}{k}\right) \\
=~ k a_0 + \sum_{n = 1}^{\infty} a_n \left(\sum_{\ell = 0}^{k - 1} \cos\left\{2 \pi n \left(x + \frac{\ell}{k}\right)\right\}\right) + \sum_{n = 1}^{\infty} b_n \left(\sum_{\ell = 0}^{k - 1} \sin\left\{2 \pi n \left(x + \frac{\ell}{k}\right)\right\}\right) .
\end{multline*}
It follows from Lemma \ref{l1} that for all $x \in \R \setminus \frac{1}{k} \Z$:
\begin{multline*}
\varphi(x) + \varphi\left(x + \frac{1}{k}\right) + \dots + \varphi\left(x + \frac{k - 1}{k}\right) \\
=~ k a_0 + \sum_{\begin{subarray}{c} n \geq 1 \\ n \equiv \modd{0} {k}\end{subarray}} k a_n \cos(2 \pi n x) + \sum_{\begin{subarray}{c} n \geq 1 \\ n \equiv \modd{0} {k}\end{subarray}} k b_n \sin(2 \pi n x) ,
\end{multline*}
which amounts to:
\begin{multline}\label{eq6}
\varphi(x) + \varphi\left(x + \frac{1}{k}\right) + \dots + \varphi\left(x + \frac{k - 1}{k}\right) \\
=~ k a_0 + \sum_{n = 1}^{\infty} k a_{k n} \cos(2 \pi k n x) + \sum_{n = 1}^{\infty} k b_{k n} \sin(2 \pi k n x)
\end{multline}
($\forall x \in \R \setminus \frac{1}{k} \Z$). \\
But on the other hand, according to \eqref{eq2}, we have for all $x \in (0 , \frac{1}{k})$:
\begin{multline*}
\varphi(x) + \varphi\left(x + \frac{1}{k}\right) + \dots + \varphi\left(x + \frac{k - 1}{k}\right) \\
=~ \log\Gamma(x) + \log\Gamma\left(x + \frac{1}{k}\right) + \dots + \log\Gamma\left(x + \frac{k - 1}{k}\right) \\
=~ \frac{k - 1}{2} \log(2 \pi) + \left(\frac{1}{2} - k x\right) \log k + \log\Gamma(k x) \\
=~ \frac{k - 1}{2} \log(2 \pi) - (\log k) B_1(\langle k x\rangle) + \varphi(k x) . 
\end{multline*}
But since the two functions $x \mapsto \varphi(x) + \varphi(x + \frac{1}{k}) + \dots + \varphi(x + \frac{k - 1}{k})$ and $x \mapsto \frac{k - 1}{2} \log(2 \pi)$ $- (\log{k}) B_1(\langle k x\rangle) + \varphi(k x)$ (which are defined on $\R \setminus \frac{1}{k} \Z$) are both $\frac{1}{k}$-periodic, we have more generally for all $x \in \R \setminus \frac{1}{k} \Z$:
$$
\varphi(x) + \varphi\left(x + \frac{1}{k}\right) + \dots + \varphi\left(x + \frac{k - 1}{k}\right) ~=~ \frac{k - 1}{2} \log(2 \pi) - (\log{k}) B_1\left(\langle k x\rangle\right) + \varphi(k x) .
$$
Using the Fourier expansions \eqref{fexpb1} and \eqref{eq3}, it follows that for all $x \in \R \setminus \frac{1}{k} \Z$:
\begin{multline}\label{eq7}
\varphi(x) + \varphi\left(x + \frac{1}{k}\right) + \dots + \varphi\left(x + \frac{k - 1}{k}\right) \\
=~ \frac{k - 1}{2} \log(2 \pi) + a_0 + \sum_{n = 1}^{\infty} a_n \cos(2 \pi k n x) + \sum_{n = 1}^{\infty} \left(b_n + \frac{\log k}{\pi n}\right) \sin(2 \pi k n x)
\end{multline}
Note that the right hand sides of \eqref{eq6} and \eqref{eq7} both represent the Fourier expansion of the $\frac{1}{k}$-periodic function $x \mapsto \varphi(x) + \varphi(x + \frac{1}{k}) + \dots + \varphi(x + \frac{k - 1}{k})$. So, according to the uniqueness of the Fourier expansion, we have:
\begin{equation}\label{eq8}
k a_0 ~=~ \frac{k - 1}{2} \log(2 \pi) + a_0 ~~,~~ k a_{k n} ~=~ a_n ~~~~~~ (\forall n \geq 1)
\end{equation}
and
\begin{equation}\label{eq9}
k b_{k n} ~=~ b_n + \frac{\log k}{\pi n} ~~~~~~~~~~ (\forall n \geq 1)
\end{equation}
The relations \eqref{eq8} immediately follow from \eqref{coeff-a_n}, so they don't give us any new information. Actually, the information about the $b_n$'s follows from \eqref{eq9}. Indeed, by taking $n = 1$ in \eqref{eq9}, we get:
$$
b_k ~=~ \frac{\log k}{\pi k} + \frac{b_1}{k} ~=~ \frac{\log k}{\pi k} + \frac{\boldsymbol{\eta}}{k} .
$$
Since this holds for any positive integer $k$, then we have:
\begin{equation}\label{coeff-b_n}
b_n ~=~ \frac{\log n}{\pi n} + \frac{\boldsymbol{\eta}}{n} ~~~~~~~~ (\forall n \geq 1)
\end{equation}
\textbf{Conclusion:}\\[1mm]
By replacing in \eqref{eq3} the values of the $a_n$'s and $b_n$'s, which are previously obtained in \eqref{coeff-a_n} and \eqref{coeff-b_n}, we get for all $x \in \R \setminus \Z$:
\begin{equation}\label{four_log_gamma}
\varphi(x) ~=~ \frac{1}{2} \log(2 \pi) + \sum_{n = 1}^{\infty} \frac{\cos(2 \pi n x)}{2 n} + \sum_{n = 1}^{\infty}\left(\frac{\log n}{\pi} + \boldsymbol{\eta}\right) \frac{\sin(2 \pi n x)}{n}
\end{equation} 
Finally, taking $x \in (0 , 1)$, we have $\varphi(x) = \log\Gamma(x)$, $\sum_{n = 1}^{\infty} \frac{\cos(2 \pi n x)}{n} = - \log 2$ $- \log\sin(\pi x)$ (according to \eqref{fexplogsin}) and $\sum_{n = 1}^{\infty} \frac{\sin(2 \pi n x)}{n} = \pi \left(\frac{1}{2} - x\right)$ (according to \eqref{fexpb1}); so it follows from \eqref{four_log_gamma} that for all $x \in (0 , 1)$, we have:
$$
\log\Gamma(x) ~=~ \frac{1}{2} \log\pi + \pi \boldsymbol{\eta}\left(\frac{1}{2} - x\right) - \frac{1}{2} \log\sin(\pi x) + \frac{1}{\pi} \sum_{n = 1}^{\infty} \frac{\log n}{n} \sin(2 \pi n x) ,
$$
as required. The theorem is proved. \hfill$\blacksquare$

\medskip

As an immediate consequence of Theorem \ref{t1}, we have the following curious corollary which gives the sum of the convergent alternating series $\sum_{n = 0}^{\infty} (-1)^n \frac{\log(2 n + 1)}{2 n + 1}$.
\begin{coll}
We have:
$$
\frac{\log 1}{1} - \frac{\log 3}{3} + \frac{\log 5}{5} - \dots ~=~ \pi \log\Gamma(1/4) - \frac{\pi^2}{4} \boldsymbol{\eta} - \frac{\pi}{2}\log\pi - \frac{\pi}{4}\log 2 .
$$
\end{coll} 
\begin{proof}
It suffices to take $x = \frac{1}{4}$ in the formula of Theorem \ref{t1}.
\end{proof}
We finish this note with the following open question:\\[1mm]
\textbf{Open question:} Is it possible to express the constant $\boldsymbol{\eta}$ in terms of the known mathematical constants as $\pi , \log\pi , \log{2} , \gamma , \Gamma(1/4) , e , \dots$?


\begin{thebibliography}{9}
\bibliographystyle{plain}
\bibitem{and}  %%% Réference vérifiée.
{\sc G. E. Andrews, R. A. Askey and R. Roy.} \textit{Special Functions}, Cambridge University Press, Cambridge, 1999.
\bibitem{art}
{\sc E. Artin.} \textit{The Gamma Function}, Holt, Rinehart, and Winston, New York, 1964.   
\bibitem{god}
{\sc M. Godefroy.} \textit{La fonction gamma: théorie, histoire, bibliographie}, 
Gauthier-Villars, Paris, 1901 (in French).
\bibitem{zyg}
{\sc A. Zygmund.} \textit{Trigonometric series}, third edition, volumes 1 and 2, Cambridge University Press, 2002.
\end{thebibliography}
\end{document}